\documentclass[11pt,a4paper]{amsart}
\usepackage{amssymb}
\usepackage{hyperref}
\usepackage[british]{babel}
\usepackage{color}
\usepackage{comment}

\usepackage{titlesec}

\titleformat{\section}
  {\normalfont\Large\bfseries}{\thesection}{1em}{}
\titlespacing*{\section}{1.0ex}{6ex}{3.0ex}
\titleformat{\subsection}
  {\normalfont\large\bfseries}{\thesubsection}{1em}{}
\titlespacing*{\subsection}{1.0ex}{4.0ex}{3.0ex}
\usepackage{parskip}
\newtheorem{proposition}{Proposition}[section]
 \newtheorem{note}[proposition]{Note}

\theoremstyle{definition}
\newtheorem{definition}[proposition]{Definition}

\newtheorem{construction}[proposition]{Construction}

\def\cov{\operatorname{Cov}}

\def\N{\mathbb{N}}
\def\R{\mathbb{R}}
\def\Q{\mathbb{Q}}

\def\|{\mid}

\def\cov{\operatorname{cov}}

\title{A countable Boolean algebra that is Reichenbach's common cause complete}

\author[D.~Bure\v{s}ov\'{a}]{Dominika Bure\v{s}ov\'{a}}
\thanks{Department of Cybernetics, 
Faculty of Electrical Engineering, 
Czech Technical University in Prague,
Czech Republic,
{\em e-mail}: buresdo2@fel.cvut.cz}
\date{}                                           % no date

\begin{document}
\begin{abstract}
    The common cause completeness (CCC) is a~philosophical principle that asserts that if we consider two positively correlated events then it evokes a common cause. The principle is due to H.~Reichenbach and has been largely studied in 
    Boolean algebras and elsewhere.
    The results published so far bring about a question whether there is a~small (countable) Boolean algebra with CCC.
    In this note we construct such an example.
\end{abstract}

\maketitle

\section{Introduction}
This note contributes to the previous work~\cite{Buresova,SRS:book2013, Kitajima, Reichenbach}. In the investigation so far, the authors found that the CCC is guaranteed if the Boolean algebra is $\sigma$-complete, has the Darboux property, and the probability measure on it is non-atomic~\cite{SRS:book2013}. 
Of course, one would like to construct as simple CCC example as possible. 
We present a~countable example showing that 
the above conditions are not necessary and
the published 
% results of 
examples of CCC are far from being the
simplest.

\section{Basic notions}
The basic notions as well as their philosophical meaning are taken from~\cite{SRS:book2013, Reichenbach}. Standardly, the symbol $p(a \mid b) = \frac{p(a \land b)}{p(b)}$ denotes the conditional probability.

\begin{definition}
  Let $a,b$ in $B$. 
  We define $\cov_p(a,b) = p(a \land b) - p(a) \cdot p(b)$. 
  % \marginpar{Define before its use. DB: Uz tomu tak jest. MN: Nebylo}
\end{definition}

\begin{definition}
    Let  $B$ be a~Boolean algebra and let $p$ be a~finitely additive probability measure on it. Suppose that $a,b,c$ are distinct elements of $B$. Then we say that $c$ 
    % (and $c \neq a, c \neq b$) 
    is a~\emph{common cause} %for 
    of $a,b$ if the following four conditions are fulfilled:
    \begin{align}
\tag{RCCP1}
\label{eqn:RCCP1}
    p(a \land b \| c) &= p(a \| c) \, p(b\| c)\,,
\\[.5em]
\tag{RCCP2}
\label{eqn:RCCP2}
    p(a \land b \| c') &= p(a \| c') \, p(b\| c')\,,
\\[.5em]
\tag{RCCP3}
\label{eqn:RCCP3} 
    p(a \| c) &> p(a \| c')\,,
\\[.5em]
\tag{RCCP4}
\label{eqn:RCCP4}
   p(b \| c) &> p(b \| c')\,.
\end{align}
% Then $c$ is said to be a~\emph{common cause} of $a,b$. 
If each pair $a,b \in B$ with positive covariance (i.e., positive correlation) has a~common cause $c$, we say that $(B, p)$ is common cause complete (CCC).
\end{definition}

Recall that if $a,b$ have a common cause then their correlation $\cov_p(a,b)%>0
$ has to be positive~\cite{Reichenbach}. Observe that if $\cov_p(a,b) > 0$ then $p(a) < 1$ and $p(b) < 1$.

 \begin{definition}\label{triv} 
 If there is no positively correlated pair of events, then conditions~\ref{eqn:RCCP1},~\ref{eqn:RCCP2},~\ref{eqn:RCCP3},~\ref{eqn:RCCP4} are automatically satisfied.
Such Boolean algebras are called \textbf{trivially common cause complete}~\cite{SRS:book2013}.
\end{definition}

\section{Results}
The technical part of our study slightly overlaps with~\cite{SRS:book2013}. We simplified and clarified its formulations for the purpose of finding the most direct way in our construction.

\begin{proposition}
    Let  $B$ be a~Boolean algebra and $p$ be a~finitely additive probability measure on $B$. Let $a,b,c \in B$, $c \leq a \land b$, $\cov_p(a,b) > 0$, and $p(c) > 0$. Then $a,b,c$ fulfill the conditions~\ref{eqn:RCCP1},~\ref{eqn:RCCP3} and~\ref{eqn:RCCP4}.
\end{proposition}
\begin{proof}
    Since  $c \leq a \land b$ then the condition~\ref{eqn:RCCP1} reduces on $1=1$.
    Next, the assumptions we are to work with are $\cov_p(a,b) > 0$, $p(a) < 1$, ${p(b) < 1}$, $c \leq a \land b \leq a$. 
    In order to verify~\ref{eqn:RCCP3}, it suffices to realize that $\frac{p(a \land c)}{p(c)} = \frac{p(c)}{p(c)} =1$ whereas $\frac{p(a \land c')}{p(c')} = \frac{p(c')-p(a')}{p(c')}<1$. The condition~\ref{eqn:RCCP4} is proved analogously.

    %This means that $p(a \land c) = p(c)$ and therefore $\frac{p(a \land %c)}{p(c)} = 1$. Hence we have~\ref{eqn:RCCP3}. The %axiom~\ref{eqn:RCCP4} derives analogously.
\end{proof}

%%%%%%
\begin{proposition}\label{3.2}
  Let $B$ a Boolean algebra and let $p$ be a~probability measure on $B$. Let $\cov_p(a,b) > 0$. Then $0 < \frac{p(a \land b) -p(a) \cdot p(b)}{1+ p(a \land b) -p(a)-p(b)} \leq p (a \land b)$.
\end{proposition}

\begin{proof}
    % Since $(1- p(a)) \cdot (1- p(b)) > 0 $, we see that $p (a \land b) - p(a) \cdot p(b) < 1 + p (a \land b) - p(a) - p(b)$ and therefore $t(a,b) > 0$. Next, $p(a \land b) \cdot \bigl(  p(a \land b) - p(a) -p(b)\bigr) + p(a) \cdot p(b) >0$. Then $p(a \land b) + p(a \land b)^2 - p(a) \cdot p(a \land b) - p(b) \cdot p(a \land b) > p(a \land b) - p(a) \cdot p(b)$. 
    Since $(1- p(a)) \cdot (1- p(b)) > 0$, we see that
$$p(a \land b) - p(a) \cdot p(b) < 1 + p(a \land b) - p(a) - p(b)$$
and therefore 
$$\frac{p(a \land b) -p(a) \cdot p(b)}{1+ p(a \land b) -p(a)-p(b)}  > 0.$$
In order to prove the other inequality, observe first the inequality of basic probability theory,
$$p(a) \cdot p(b) \geq p(a \lor b) \cdot p(a \land b).$$
When we adjust the desired inequality  $\frac{p(a \land b) -p(a) \cdot p(b)}{1+ p(a \land b) -p(a)-p(b)} \leq p (a \land b)$, by multiplying with the denominator we obtain 
$$p(a \land b) + p(a \land b)^2 - p(a \land b) \cdot p(a) - p(a \land b) \land p(b) \geq p(a \land b) - p(a) \cdot p(b)\,.$$
Manipulating with this inequality, we obtain
$$p(a \land b)^2 - p(a \land b) \cdot p(a) - p(a \land b) \cdot p(b) + p(a) \cdot p(b) \geq 0\,.$$
This is equivalent to
$$\bigl(p(a) - p(a \land b)\bigr) \cdot \bigl(p(b) - p(a \land b)\bigr) \geq 0\,.$$
which is obviously true.
% This gives us the inequality
% $$p(a \land b) \cdot \bigl(p(a \land b) - p(a) - p(b)   \bigr) + p(a) \cdot p(b) \geq 0.$$
% Using the identity $p(a) + p(b) = p(a \land b) + p(a \lor b)$ we further simplify the inequality above to get
% $$ \Bigl( p (a \land b) \cdot \bigl( -p (a \lor b)  \bigr) \Bigr) + p(a) \cdot  p(b) \geq 0  .$$
% But this is the inequality we referred to above as a basic one in probability theory. 
The proof is complete.
\end{proof}
%%%%%%
\begin{note}
    The inequality in Proposition~\ref{3.2} was dealt with in the book~\cite{SRS:book2013}, however the proof presented there was erroneous in the fundamental statement 
(4.92). We present a~correct and new proof. 
\end{note}
 
%%%%%%
\begin{proposition}
  Let $\cov_p(a,b) > 0, c \leq a \land b$.
  If $c$ is such that $p(c) = \frac{p(a \land b) -p(a) \cdot p(b)}{1+ p(a \land b) -p(a)-p(b)}$ then $c$ is the common cause of $a,b$.
\end{proposition}

\begin{proof}
    We have to check~\ref{eqn:RCCP2}. Our assumption is $$p(c) -p(c)\cdot p(a) - p(c) \cdot p(b) + p(c) \cdot p (a \land b) = p(a \land b) - p(a) \cdot p(b)\,.$$ Adding $p(c)^2$ to both sides of the equation and reorganizing it, we get:
    \begin{align*}
        &p(c)^2 + p(a) \cdot p(b) - p(a)\cdot p(c) - p(b)\cdot p(c) 
        \\=\ &p(c)^2 + p(a \land b) %- p(a \land b) 
        - p(c) - p(c) \cdot p(a \land b)\,.
    \end{align*}
    This implies that 
    $$\underbrace{\bigl( p(a)- p(c)\bigr)}_{p(a \land c')} \cdot \underbrace{\bigl( p(b)- p(c)\bigr)}_{p(b \land c')} 
    = \underbrace{\bigl(1 - p(c)\bigr)}_{p(c')} \cdot \underbrace{\bigl(p(a \land b) - p(c)\bigr)}_{p(a \land b \land c')}$$ 
    % But $1-p(c) = p(c')$, $p(a) - p(c) = p(a \land c')$ and similarly for $b$ or $a \land b$ in place of~$a$. Hence we arrive to 
    % $$p(a \land c') \cdot p(b \land c') = p(c') \cdot p(a \land b \land c')$$
    % % $\frac{p(a \land b \land c')}{p(c')} = \frac{p(a \land c') \cdot p(b \land c')}{^2}$ 
    and this is~\ref{eqn:RCCP2}.
\end{proof}

%%%%%%

Let us note that finite Boolean algebras can be only trivially common cause complete (Definition~\ref{triv}). Hence a~meaningful case for looking for the CCC begins with a~countable cardinality of $B$.

\begin{proposition}\label{pip}
    There is a~countable Boolean algebra $B$ with a~finitely additive probability measure $p$ such that $(B, p)$ is non-trivially common cause complete.
\end{proposition}
\begin{construction}
    Take the interval $(0, 1]$ and the collection of finite disjoint unions of sub-intervals of $(0, 1]$ that have rational endpoints and that are left-open and right-closed. This collection forms a~Boolean algebra. Indeed, this collection is closed under the formation of unions, intersections, and complements in $(0, 1]$. Also, this Boolean algebra is countable since it has the same cardinality as $\bigcup_{i \in \N} \Q^i$ ($\Q$ denotes the rational numbers of $(0, 1]$). Take the probability measure $p$ as the restriction of the Lebesgue measure to~$B$. By the construction, $p$ is common cause complete. Indeed, consider the value $v = \frac{p(a \land b) -p(a) \cdot p(b)}{1+ p(a \land b) -p(a)-p(b)}  > 0$. So $v$ is a~rational number with $v \leq p(a \land b)$. By the definition of $v$ and $p$ it is easily seen that there is an element $c$ such that $c \leq a \lor b$ and $p(c) =v$. This completes the proof.
\end{construction}
Let us shortly comment on the result and its proof. Firstly, the same method gives us also an uncountable Boolean algebra that is non-trivially common cause complete. Indeed, the classic results of the division ring theory assert that there is a~division ring $\tilde{R}$ in the real numbers $\R$ and  $\tilde{R}$ is uncountable and dense in $\R$. In analogy with Construction~\ref{pip}, one takes for $B$ the collection of disjoint unions 
of intervals
in $\tilde{R} \cap (0,1]$.
It suffices to use the Lebesgue measure again.\\

Secondly, though we mainly addressed the Boolean line of CCC, if we stepped into quantum logics (see \cite{SRS:book2013, Gudder1} for the adequate notions on quantum logics), we see that the situation is analogous there (including the fact that finite quantum logics cannot be non-trivially CCC). We can easily construct ``finitely additive'' non-trivially CCC quantum logics, too. We achieve this goal by a~construction called a~horizontal sum (see~\cite{Kalmbach} for the adequate notion in quantum logics) of arbitrarily many copies of the example constructed in Proposition~\ref{pip}. 
Hence in quantum logics we can easily use our Proposition~\ref{pip} for constructing examples of an arbitrarily large (infinite) cardinality.

 \textbf{Acknowledgements.} 
 The author would like to thank her colleagues Pavel Pták, Mirko Navara, and Miroslav Korbelář for professional suggestions that improved the text.\\
 {This work was supported by the CTU institutional support RVO13000.}

\textbf{Data availability statement.} The manuscript has no associated data.

\textbf{Conflict of interest statement.} The author states that there is no conflict of interest.

% \section*{Affiliation}
% {\sc Dominika Bure\v{s}ov\'{a}}

% Department of Cybernetics

% Czech Technical University, Faculty of Electrical Engineering

% 166 27  Prague 6

% Czech Republic

% {\em e-mail}: buresdo2@fel.cvut.cz

\end{document}